\documentclass[a4paper, oneside,11 pt]{article}

\usepackage[latin1]{inputenc}
\usepackage[T1]{fontenc}
\usepackage[english]{babel}

\usepackage[all]{xy}
\usepackage{verbatim}
\usepackage{amsmath}
\usepackage{amsthm}
\usepackage{graphicx}
\usepackage{amssymb}
\usepackage{epstopdf}
\usepackage{url}
\usepackage{amsfonts}
\usepackage{todonotes}

\newcommand{\vp}{\varphi}

  \newcommand{\M}{{\mathcal M}}
    \newcommand{\W}{{\mathcal W}}
   \newcommand{\U}{{\mathcal U}}
   \newcommand{\V}{{\mathcal V}}
      \renewcommand{\H}{{\mathbb H}}
            \newcommand{\N}{{\mathbb N}}
             \renewcommand{\S}{{\mathbb S}}
      \newcommand{\HH}{{\mathcal H}}
\newcommand{\br}{\mathbb{R}}

\newcommand{\te}{\theta}
\newcommand{\vte}{\vartheta}
\renewcommand{\a}{\alpha}
\newcommand{\g}{\gamma}
\newcommand{\s}{\sigma}

\newcommand{\hh}{h}
\newcommand{\vol}{ {\rm vol}}
\newcommand{\e}{\varepsilon}

\renewcommand{\r}{\rho}

\renewcommand{\(}{\left(}
\renewcommand{\)}{\right)}
\renewcommand{\[}{\left[}
\renewcommand{\]}{\right]}

\renewcommand{\>}{\rangle}

\newcommand{\ra}{\rightarrow}
\newcommand{\supp}{{\rm supp}}
\newcommand{\dist}{{\rm dist}}
\newcommand{\one}{\mathbf{1}}
\newtheorem{thm}{Theorem}
\newtheorem{lem}[thm]{Lemma}
\newtheorem{cor}[thm]{Corollary}

\newtheorem{defi}[thm]{Definition}

\def\be{\begin{equation}}
\def\ee{\end{equation}}
\def\bea{\begin{eqnarray}}
\def\eea{\end{eqnarray}}

\numberwithin{thm}{section}
\numberwithin{equation}{section}

\title{Asymptotic quantization for probability measures on Riemannian manifolds}

\author{
  Mikaela Iacobelli\
   \thanks{University of Roma La Sapienza, Dipartimento Guido Castelnuovo, Piazzale Aldo Moro 5, 00185 Roma, ITALY. Email: \textsf{iacobelli@mat.uniroma1.it}}
   \thanks{Ecole Polytechnique, Centre de math\'ematiques Laurent Schwartz, 91128 Palaiseau Cedex, FRANCE.}
   }

\begin{document}

\maketitle

%

\begin{abstract}
In this paper we study the quantization problem for probability measures on Riemannian manifolds. Under a suitable assumption on the growth at infinity of the measure we find asymptotic estimates for the quantization error, generalizing the results on $\mathbb{R}^d.$ Our growth assumption depends on the curvature of the manifold  and reduces, in the flat case, to a moment condition. We also build an example showing that our hypothesis is sharp.

\medskip
{\bf R\'esum\'e:} Dans ce travail nous \'etudions le probl\`eme de quantification des mesures sur les vari\'et\'es Riemanniennes. Sous des hypoth\`eses  convenables sur la croissance de la mesure \`a l'infini, nous obtenons
des estim\'ees asymptotiques pour l'erreur de quantification. Ceci g\'en\'eralise les r\'esultats connus dans $\mathbb{R}^d$.
Notre hypoth\`ese de croissance d\'epend de la courbure de la vari\'et\'e et,  dans le cas plat, correspond \`a un contr\^ole sur les moments. Nous construisons aussi un exemple  pour montrer que notre hypoth\`ese est n\'ecessaire.
\end{abstract}

%
\tableofcontents

\section{Introduction}

  The problem of quantization of a $d$-dimensional probability distribution deals with constructive methods to find atomic probability measures supported on a finite number of points, which best approximate a given diffuse probability measure. The quality of this approximation is usually measured in terms of the Wasserstein metric,
  and up to now this problem has been studied in the flat case and on compact manifolds.
 
 The quantization problem arises in several contexts and has  applications in signal compression, pattern recognition, speech recognition, stochastic processes, numerical integration, optimal location of service centers, and kinetic theory. For a detailed exposition and a complete list of references, we refer to the monograph \cite{GL} and references therein.
In this paper we study it for probability measures on general Riemannian manifolds. Apart from its own interest, this has several natural applications. 

To mention one, in order to find a good approximation of a convex body by polyhedra one may look for the best approximation of the curvature measure of the convex body by discrete measures \cite{B}. 

To give another natural motivation, let us present the so-called \emph{location problem}. If we want to plan the location of a certain number of grocery stores to meet the demands of the population in a city, we need to chose the optimal location and size of the stores with respect to the distribution of the population.
%
%
%
The classical case on $\br^d$ corresponds to the situation of a city on a flat land. 
Now consider the possibility that the geographical region, instead of being flat, is situated either at the bottom of a valley, or at a pass in the mountains. Then the Wasserstein distance reflects this geography, by depending on the distance $d$ as measured along a spherical cap in the case of the valley or along a piece of a saddle in the case of the mountain pass. It follows by our results that for a city in a valley or on a mountain top the optimal location problem converge as in the flat case, while for a city located on a pass in the mountains the effect of negative curvature badly influences the quality of the approximation. Hence, our results display how geometry and geography can affect the optimal location problem.
\\

 We now introduce the setting of the problem. Let $(\M, g)$ be a complete Riemannian manifold, and fixed $r\ge 1$, consider $\mu$ a probability measure on $ \M$.
 Given $N$ points $x^{1}, \ldots, x^{N} \in \M,$ one wants to find the  best approximation of $\mu,$ in the Wasserstein distance $W_r$, by a convex combination of Dirac masses centered at $x^{1}, \ldots, x^{N}.$ Hence one minimizes
$$
\inf \bigg\{ W_r\bigg(\sum_im_i \delta_{x^i}, \mu\bigg)^r\,:\, m_1, \ldots, m_N\ge0, \  \sum_im_i=1
\bigg\},
$$
with
$$
W_r(\nu_1, \nu_2):=\inf\bigg\{
\biggl(\int_{\M\times\M}d(x,y)^rd\gamma(x,y)\biggr)^{1/r}\,:\,
(\pi_1)_\#\gamma=\nu_1,\  (\pi_2)_\#\gamma=\nu_2\bigg\},
$$
where $\gamma$ varies among all probability measures on $\M \times \M$, $\pi_i: \M \times \M \to \M$ ($i=1,2$) denotes the canonical projection onto the $i$-th factor, and $d(x,y)$ denotes the Riemannian distance.
The best choice of the masses $m_i$ is explicit and can be expressed in terms of the so-called {\it Voronoi cells} \cite[Chapter 1.4]{GL}. Also, as shown for instance in \cite[Chapter 1, Lemmas 3.1\ and 3.4]{GL},
the following identity holds:
$$
\inf \bigg\{ W_r\bigg(\sum_im_i \delta_{x^i}, \mu\bigg)^r\,:\, m_1, \ldots, m_N\ge0, \  \sum_im_i=1
\bigg\}
=F_{N,r} (x^{1}, \ldots, x^{N}) ,
$$
where
 $$
F_{N,r} (x^{1}, \ldots, x^{N}) := \int_{\M} \underset{1\le i \le N}{\mbox{min}} d(x^i,y)^r\,d\mu(y).
$$
Hence, the main question becomes:  Where are the ``optimal points'' $(x^{1}, \ldots, x^{N})$ located?
To answer to this question, at least in the limit as $N\to \infty$, let us first introduce some definitions.
\begin{defi}
Let $\mu$ be a probability a probability measure on $\M$, $N\in \mathbb{N}$ and $r\ge1$. Then, we define the \emph{$N$-th quantization error of order $r$}, $V_{N,r}(\mu)$ as follows:
\be
\label{eq:V_{N,r}}
V_{N,r}(\mu):= \underset{\alpha \subset \M : |\alpha|\le N}{\inf}  \int_{\M} \underset{a\in \alpha}{\min} \,d(a,y)^r\,d\mu(y),
\ee
where $|\alpha|$ denotes the cardinality of a set $\alpha.$

Let us notice that, being the functional $F_{N,r}$ decreasing with respect to the number of points $N,$ an equivalent definition of $V_{N,r}$ is:
$$
V_{N,r}(\mu):=\underset{x^1, \ldots, x^N \in \M}{\inf} F_{N,r}(x^1, \ldots, x^N).
$$
\end{defi}

Let us observe that the above definitions make sense for general positive measures with finite mass. In the sequel we will sometimes consider this class of measures in order to avoid renormalization constants.

A quantity that plays an important role in our result is the following:
\begin{defi} \label{def:Qrd}Let $dx$ be the Lebesgue measure and  $\chi_{[0,1]^d}$ the characteristic function of the unit cube ${[0,1]^d}.$ We set 
$$
Q_r\([0,1]^d\):= \underset{N\ge1}{\inf} N^{r/d}V_{N,r}\(\chi_{[0,1]^d} dx\).
$$
\end{defi}
As proved in \cite[Theorem 6.2]{GL}, $Q_r\([0,1]^d\)$ is a positive constant.
The following result describe the asymptotic distribution of the minimizing configuration in $\mathbb{R}^d$,
answering to our question in the flat case
(see \cite{BW} and \cite[Chapter 2, Theorems 6.2 and 7.5]{GL}): 

\begin{thm}
\label{thm: R^d}
Let $\mu=h\,dx+\mu^s$ be a probability measure on $\mathbb{R}^d$, where $\mu^s$ denotes the singular part of $\mu$. Assume that $\mu$ satisfies
\be
\label{eq:moment}
\int_{\mathbb{R}^d}|x|^{r+\delta}\,d\mu(x)< \infty.
\ee
Then 
\be\label{close1}
\underset{N\to\infty}{\lim}N^{r/d} V_{N,r}(\mu)=Q_r\([0,1]^d\)\,\biggl(\int_{\mathbb R^d}h^{d/(d+r)}\,dx\biggr)^{(d+r)/d}.
\ee
In addition, if $\mu^s\equiv 0$ and
$x^{1}, \ldots, x^{N}$ minimize the functional $F_{N,r}: (\mathbb{R}^d)^N \ra \mathbb{R}^+$, then 
\be\label{close}
\frac{1}{N}\sum_{i=1}^N \delta_{x^i} \rightharpoonup \frac{h^{d/{d+r}}}{\int_{\mathbb{R}^d} h^{d/{d+r}}(y)dy}\,dx \qquad \text{as $N \to \infty.$}
\ee
  \end{thm}

The first statement in the above theorem has been generalized to the case of absolutely continuous probability measures on compact Riemannian manifolds in \cite{B}. The aim of this paper is twofold: we first give a shorter proof of Theorem \ref{thm: R^d} for general probability measures on compact manifolds, and then we extend it to arbitrary  measures on non-compact manifolds.
As we shall see, passing from the compact to the non-compact setting presents nontrivial difficulties. Indeed, while the compact case relies on a localization argument that allows one to mimic the proof in $\mathbb{R}^d$, the non-compact case requires additional new ideas. In particular one needs to find a suitable analogue of the moment condition \eqref{eq:moment} to control the growth at infinity of our given probability measure. We will prove that the needed growth assumption depends on the curvature of the manifold (and more precisely, on the size of the differential of the exponential map).\\

To state in detail our main result we need to introduce some notation:
given a point $x_0 \in \M$, we can 
consider polar coordinates $(\r,\vte)$ on $T_{x_0}\M\simeq \br^d$ induced by the constant metric $g_{x_0}$,
where $\vartheta$ denotes a vector on the unit sphere $\S^{d-1}$.
Then, we can define the following quantity that measures the size of the differential of the exponential map when restricted to a sphere 
$\S^{d-1}_\rho\subset T_{x_0}\M$ of radius $\rho$:
\be
\label{eq:Arho}
A_{x_0}(\rho):=\sup_{v\in \S^{d-1}_\rho,\,w\in T_{v}\S^{d-1}_\rho,\, |w|_{x_0}=\rho} \Bigl| d_v \exp_{x_0}[w]\Bigr|_{\exp_{x_0}(v)},
\ee
To prove asymptotic quantization, we shall impose an analogue of  \eqref{eq:moment} which involves the above quantity.
\begin{thm}\label{thm:main}
Let $(\M,g)$ be a complete Riemannian manifold without boundary, and 
 let $\mu=\hh \,d\vol+\mu^s$ be a probability measure on $\M$.
 Assume there exist a point $x_0 \in \M$ and $\delta>0$ such that 
\be
\label{eq:moment M}
 \int_{\M} d(x,x_0)^{r+\delta}\,d\mu(x)+  \int_{\M} A_{x_0}\bigl({d(x,x_0)}\bigr)^r  \,d\mu(x)<\infty.
 \ee
Then \eqref{close1} holds.
\end{thm} 

Once this theorem is obtained, by the very same argument as in \cite[Proof of Theorem 7.5]{GL}
one gets the following:

\begin{cor}
Let $(\M,g)$ be a complete Riemannian manifold without boundary, $\mu=\hh \,d\vol$ an absolutely continuous probability measure on $\M$ and let $x^1, \ldots, x^N$ minimize the functional $F_{N,r}: \M^{\otimes N} \to \br^+.$ Assume there exist a point $x_0 \in \M$ and $\delta>0$ for which \eqref{eq:moment M} is satisfied. Then
\eqref{close} holds.
\end{cor}
Notice that the quantity $A_{x_0}$ is related to the curvature of $\M,$ being linked to the size of the Jacobi fields (see for instance \cite[Chapter 10]{L}).
In particular, if $\M=\H^d$ is the hyperbolic space then $A_{x_0}(\r)=\sinh \r$,
while on $\br^d$ we have $A_{x_0}(\r)=\r$. Hence the above condition on $\H^d$ reads as
$$
\bigg(1+ \int_{\H^d} d(x,x_0)^{r+\delta}\,d\mu(x)+  \int_{\H^d} \sinh\bigl({d(x,x_0)}\bigr)^r  \,d\mu(x)u \bigg) \approx \int_{\H^d} e^{r\,d(x,x_0)} \,d\mu(x),
$$
and on $\br^d$ as
$$
\bigg(1+ \int_{\br^d} d(x,x_0)^{r+\delta}\,d\mu(x)+  \int_{\br^d} {d(x,x_0)}^r  \,d\mu(x)u \bigg) \approx \int_{\br^d}d(x,x_0)^{r+\delta}\,d\mu(x).
$$
Hence \eqref{eq:quantiz noncpt} holds on $\H^d$ for any probability measure $\mu$ satisfying 
$$
\int_{\H^d} e^{r\,d(x,x_0)} \,d\mu(x)<\infty
$$
for some $x_0 \in \H^d$, while on $\br^d$ we only need the finiteness of some $(r+\delta)$-moments of $\mu$,
therefore recovering the assumption in Theorem \ref{thm: R^d}.
More in general, thanks to Rauch Comparison Theorem \cite[Theorem 11.9]{L}, the size of the Jacobi fields 
on a manifold $\M$ with sectional curvature bounded from below by $-K$ ($K\geq 0$) is controlled by the 
Jacobi fields on the hyperbolic space with sectional curvature $-K$.
Hence in this case $$A_{x_0}(\r) \leq \sinh(Kr) \approx e^{Kr},$$ and 
Theorem
\ref{thm:main} yields the following:
\begin{cor}
Let $(\M,g)$ be a complete Riemannian manifold without boundary, and 
 let $\mu=\hh \,d\vol+\mu^s$ be a probability measure on $\M$.
 Assume that the sectional curvature of $\M$ is bounded from below by $-K$  for some $K\geq 0$,
 and that there exist a point $x_0 \in \M$ and $\delta>0$ such that 
$$
 \int_{\M} d(x,x_0)^{r+\delta}\,d\mu(x)+  \int_{\M} e^{Kr \,d(x,x_0)}  \,d\mu(x)<\infty.
$$
Then \eqref{close1} holds.
In addition, if $\mu^s\equiv 0$ and
$x^{1}, \ldots, x^{N}$ minimize the functional $F_{N,r}: (\mathbb{R}^d)^N \ra \mathbb{R}^+$, then 
\eqref{close} holds.
\end{cor}

Finally, we show that the moment  condition \eqref{eq:moment} required on $\br^d$ is not sufficient to ensure the validity 
of the result on $\H^d$. Indeed we can provide the following counter example on $\mathbb H^2.$

\begin{thm}
\label{thm:counter}
There exists a measure $\mu$ on $\mathbb H^2$ such that
$$
\int_{\H^2} d(x,x_0)^p\,d\mu<\infty \qquad \forall\,p >0,\,\forall \,x_0 \in \H^2,
$$
but
$$
N^{r/2} V_{N,r}(\mu) \to \infty \qquad \text{as }N\to \infty.
$$
\end{thm}

The paper is structured as follows: first, in Section \ref{sec:compact} we prove Theorem \ref{thm:main} for compactly supported probability measures. Then, in Section \ref{sec:non-compact} we deal with the non-compact case concluding the proof of Theorem \ref{thm:main}. Finally, in Section \ref{sec:counter} we prove Theorem \ref{thm:counter}.
\section{Proof of Theorem \ref{thm:main}: the compact case} 
\label{sec:compact}

This section is concerned with the study of asymptotic quantization for probability distributions on compact Riemannian manifolds as the number $N$ of points  tends to infinity. Although the problem depends a priori on the global geometry of the manifold (since $V_{N,r}$ involves 
the Riemannian distance), we shall now show how a localization argument allows us to prove the result.

\subsection{Localization argument}

Let $(\M, g)$ be a complete Riemannian manifold without boundary and let $\mu$ be a probability measure on $\M.$ We consider $\{ \U_i, \vp_i\}_{i\in I}$ an atlas covering $\M$, and $\vp_i: \W_i \to \br^d$ smooth charts, where $\W_i \supset\supset \U_i$ for all $i \in I.$ 
As we shall see, in order to be able to split our measure as a sum of measures supported on smaller sets,
we want to avoid the mass to concentrate on the boundary of the sets $\U_i$.
Hence, up to slightly changing the sets $\U_i$, we may assume that
\be
\label{eq:no boundary}
\mu(\partial \U_i)=0\qquad \forall\,i \in I.
\ee
We want to cover $\M$ with an atlas of disjoint sets, up to sets of $\mu$-measure zero. To do that
we define
$$
\V_i:=\U_i\setminus \biggl(\bigcup_{j=1}^{i-1}\U_j\biggr).
$$
Notice that we still have $\V_i\subset\subset \W_i$.

Given an open subset of $\br^d,$ by \cite[Lemma 1.4.2]{C}, we can cover it with a countable partition of half-open disjoint cubes  such that the maximum length of the edges is a given number $\delta$. We now apply this observation to each open subset $\vp_i(\overset{\circ}{\V}_i) \subset\br^d$ and we cover it with a family $\mathcal G_i$ of half-open cubes $\{Q_{i,j}\}_{j \in \mathbb{N}}$ with edges of length $\ell_j\le \delta.$

\begin{figure}[h]
\centerline{\includegraphics[scale=0.5]{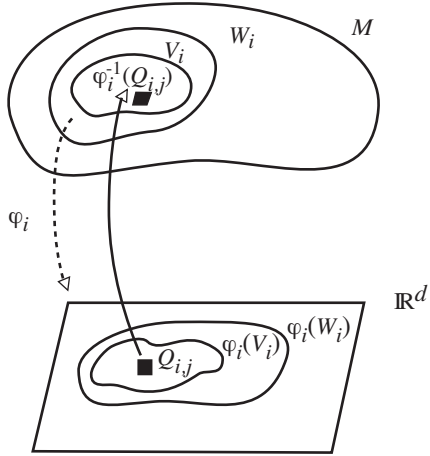}}
\caption{\label{fig1} We use the map $\varphi_i^{-1}:\varphi_i(\W_i)\subset \br^d \to \W_i\subset \M$ to send the partition in cubes $Q_{i,j}$
of $\varphi_i(\V_i)$ on $\M$.}
\end{figure}

We notice that the ``cubes'' $\vp_i^{-1}(Q_{i,j})\subset \M$ are disjoint and 
$$
\bigcup_{i \in I} \bigcup_{Q_{i,j}\in \mathcal G_i} Q_{i,j}=\M\setminus \biggl(\bigcup_i \partial \U_i\biggr)
$$
Since by \eqref{eq:no boundary} the set $\cup_i \partial \U_i$ has zero $\mu$-measure, we can decompose the 
measure $\mu$ as 
$$
\mu=\sum_{i \in I} \mu \one_{\V_i}=\sum_{i \in I} \sum_{Q_{i,j}\in \mathcal G_i}\mu \one_{\V_i\cap \vp_i^{-1}(Q_{i,j}) }.
$$
We now set
$$
\a_{ij}:=\int_{\V_i\cap \vp_i^{-1}(Q_{i,j}) }d\mu,\qquad 
\mu_{ij}:=\frac{\mu \one_{\V_i\cap \vp_i^{-1}(Q_{i,j}) }}{\a_{ij}},
$$
so that
$$
\mu=\sum_{ij}\a_{ij}\,\mu_{ij},\qquad \int_{\M}d\mu_{ij}=1,\qquad \supp (\mu_{ij})\subset \V_i\cap \vp_i^{-1}(Q_{i,j}),
$$
where, to simplify the notation, in the above formula the indices $i,j$ implicitly run over $i \in I, Q_{i,j}\in \mathcal G_i$. We will keep using this convention also later on.\\

The idea is now the following: by choosing $\delta$ small enough, each measure $\mu_{ij}$ is supported on a
very small set where the metric is essentially constant and allows us to reduce ourselves to the flat case 
and apply Theorem \ref{thm: R^d} to each of these measures.
A ``gluing argument'' then gives the result when
 $\mu=\sum_{ij}\alpha_{ij}\mu_{ij}$ is  compactly supported, $\alpha_{ij}\neq 0$
for at most finitely many indices, and
 $\mu_{ij}$  has constant density on $\vp_i^{-1}(Q_{i,j})$.
Finally, an approximation argument 
yields the result for general compactly supported measures.

\subsection{The local quantization error}

The goal of this section is to understand the behavior of $V_{N,r}(\mu)$
when 
\be
\label{eq:mu 1}
\mu=\lambda \one_{\vp^{-1}(Q)}\,d\vol,
\ee
where $\lambda:=\frac{1}{\vol(\vp^{-1}(Q))}$ (so that $\mu$ has mass $1$),
$Q$ is a $\delta$-cube in $\br^d$, $\vp:\W\to \br^d$ is a diffeomorphism defined on a neighborhood
$\W\subset \M$ of $\vp^{-1}(Q)$.

We observe that, in the computation of $V_{N,r}(\mu)$, if the size of the cube is sufficiently small
then we can assume that all the points belong to a $K\delta$-neighborhood of $\vp^{-1}(Q)$,
with $K$ a large universal constant, that we denote by $\mathcal Z_{K\delta}$.
Indeed, if $\dist(b,\vp^{-1}(Q))>K\delta$ then
$$
\dist(x,b)>\dist(x,y)\qquad \forall\,x,y \in \vp^{-1}(Q),
$$
which implies that, in the definition of $V_{N,r}(\mu)$, it is better to substitute $b$ with an arbitrary point inside $\vp^{-1}(Q)$.
Notice also that, if $\delta$ is small enough, $\mathcal Z_{K\delta}$ will be contained in the chart $\W$.

Hence, denoting by $\beta$ a family of $N$ points inside a $\mathcal Z_{K\delta}$, and by $\alpha$
a family of $N$ points inside $\vp(\mathcal Z_{K\delta})$, we have
\be
\label{eq:VNcube}
\begin{split}
V_{N,r}(\mu)&=\inf_\beta \int_{\vp^{-1}(Q)} \min_{b \in \beta}d(y,b)^r\,d\mu(y)\\
&=\lambda\,\inf_\beta \int_{\vp^{-1}(Q)} \min_{b \in \beta}d(y,b)^r\,d\vol(y)\\
&=\lambda \,\inf_\beta \int_{Q} \min_{a \in \alpha}d\bigl(\vp^{-1}(x),\vp^{-1}(a)\bigr)^r\,\sqrt{\det\,g_{k\ell}(x)}\,dx.
\end{split}
\ee

We now begin by showing that $d\bigl(\vp^{-1}(x),\vp^{-1}(a)\bigr)$ can be approximated with a constant metric.
Recall that $\delta$ denotes the size of the cube $Q$.
Also, we use the notation $g_{k\ell}$ to denote the metric in the chart, that is
\be
\label{eq:metric}
\sum_{k\ell}g_{k\ell}(x)v^kv^\ell:=g_{\vp^{-1}(x)}\Bigl(d\vp^{-1}(x)[v],d\vp^{-1}(x)[v]\Bigr),\qquad \forall\,x \in \vp(\W),\,v \in \br^d.
\ee
\begin{lem}
\label{lem:metric}
Let $p$ be the center of the cube $Q$ and let
$A$ be the matrix with entries $A_{k\ell}:=g_{k\ell}(p)$.
There exists a universal constant $\hat C$ such that, for all $x \in Q$ and $a \in \vp(\mathcal Z_{K\delta})$, it holds
$$
(1-\hat C\delta)\,\<A(x-a),x-a\> \leq d\bigl(\vp^{-1}(x),\vp^{-1}(a)\bigr)^2 \leq (1+\hat C\delta)\,\<A(x-a),x-a\>.
$$
\end{lem}

\begin{proof}
We begin by recalling that
\footnote{Recall that there are two equivalent definition of the distance between two points:
$$
d(x,y)=\inf_{\substack{ \gamma(0)=x, \\ \gamma(1)=y}}
\int_0^1 \sqrt{g_{\gamma(t)}\bigl( \dot\gamma(t),\dot\gamma(t)\bigr)}\,dt=\inf_{\substack{ \gamma(0)=x, \\ \gamma(1)=y}}\sqrt{
\int_0^1 g_{\gamma(t)}\bigl( \dot\gamma(t),\dot\gamma(t)\bigr)\,dt}.
$$
In this paper we will make use of both definitions.
}
$$
d\bigl(\vp^{-1}(x),\vp^{-1}(a)\bigr)^2=\inf_{\substack{ \gamma(0)=\vp^{-1}(x), \\ \gamma(1)=\vp^{-1}(a)}}
\int_0^1 g_{\gamma(t)}\bigl( \dot\gamma(t),\dot\gamma(t)\bigr)\,dt.
$$
Let $\bar\gamma:[0,1]\to \M$ denote a minimizing geodesic.\footnote{Notice that the hypothesis of completeness on $\M$ ensures the existence of minimizing geodesics.}
Then the speed of $\bar\gamma$ is constant and equal to the distance between the two points, that is
\be
\label{eq:dot gamma}
\|\dot{\bar\gamma}(t)\|_{g}:=\sqrt{g_{\bar\gamma(t)}\bigl( \dot{\bar\gamma}(t),\dot{\bar\gamma}(t)\bigr)}=d\bigl(\vp^{-1}(x),\vp^{-1}(a)\bigr).
\ee
We can bound from above $d\bigl(\vp^{-1}(x),\vp^{-1}(a)\bigr)$ by choosing a curve $\gamma$ obtained
by the image via $\vp^{-1}$ of a segment:
$$
d\bigl(\vp^{-1}(x),\vp^{-1}(a)\bigr)^2\leq \int_0^1 g_{\sigma(t)}\bigl( \dot\sigma(t),\dot\s(t)\bigr)\,dt,
\qquad \sigma(t):=\vp^{-1}\bigl((1-t)x+ta\bigr).
$$
Observe that this formula makes sense since $(1-t)x+ta \in \vp(\W)$ provided $\delta$ is sufficiently small.

Since
\be
\label{eq:dist}
\sqrt{\int_0^1 g_{\sigma(t)}\bigl( \dot\sigma(t),\dot\s(t)\bigr)\,dt} \leq C'|x-a|
\ee
for some universal constant $C'$,
combining \eqref{eq:dot gamma} and \eqref{eq:dist}
we deduce that
$$
\|\dot{\bar\gamma}(t)\|_{g} \leq C'|x-a| \leq C''\delta \qquad\forall\,t \in [0,1].
$$
In particular 
$$
d(\bar\gamma(t),x)=d(\bar\gamma(t),\bar\g(0)) \leq C''\delta\qquad \text{for all $t \in [0,1]$,}
$$
which implies that 
 $\bar\gamma$ belongs to the $K\delta$-neighborhood of $\vp^{-1}(Q)$, that is $\bar\gamma \subset \mathcal Z_{C''\delta}$.

Thanks to this fact we deduce that in the definition of the distance we can use only curves contained inside $\mathcal Z_{C''\delta}$.
Since $\mathcal Z_{C''\delta}\subset \W$ for $\delta$ sufficiently small, all such curves can be seen as the image through $\vp^{-1}$
of a curve contained inside $\vp(\W)\subset \br^d$.
Notice that, by \eqref{eq:metric}, if 
$$
\sigma(t):=\vp(\gamma(t))=\bigl(\sigma^1(t),\ldots,\sigma^n(t)\bigr) \in \br^d
$$
then
$$
g_{\gamma(t)}\bigl( \dot\gamma(t),\dot\gamma(t)\bigr)
=\sum_{k\ell} g_{k\ell}(\sigma(t))\dot\sigma^k(t)\dot\sigma^\ell(t),
$$
therefore
\begin{align*}
d\bigl(\vp^{-1}(x),\vp^{-1}(a)\bigr)^2&=\inf_{\substack{\gamma(0)=\vp^{-1}(x),\\ \gamma(1)=\vp^{-1}(a)}}\int_0^1 g_{\gamma(t)}\bigl( \dot\gamma(t),\dot\gamma(t)\bigr)\,dt\\
&=\inf_{\substack{\gamma(0)=\vp^{-1}(x),\\ \gamma(1)=\vp^{-1}(a),\\ \gamma \subset \mathcal Z_{C''\delta}}}
\int_0^1 g_{\gamma(t)}\bigl( \dot\gamma(t),\dot\gamma(t)\bigr)\,dt\\
&=\inf_{\substack{\sigma(0)=x,\, \sigma(1)=a,\\ \sigma\subset \vp(\mathcal Z_{C''\delta})}}
\int_0^1 \sum_{k\ell} g_{k\ell}(\sigma(t))\dot\sigma^k(t)\dot\sigma^\ell(t)\,dt\\
&\leq \bigl(1+\hat C \delta\bigr) \inf_{\substack{\sigma(0)=x,\, \sigma(1)=a,\\ \sigma\subset \vp(\mathcal Z_{C''\delta})}}
\int_0^1 \sum_{k\ell} A_{k\ell}\dot\sigma^k(t)\dot\sigma^\ell(t)\,dt,
\end{align*}
where in the last inequality we used that,
 by the Lipschitz regularity of the metric
and the fact that $g_{k\ell}$ is positive definite, we have
$$
\sum_{k\ell}g_{k\ell}(z)v^kv^\ell  \leq (1+\hat C\delta) \sum_{k\ell}A_{k\ell}v^kv^\ell \qquad \forall\,z  \in \vp(\mathcal Z_{C''\delta}),\qquad \forall\,v \in \br^d.
$$
Using now that the minimizer for the problem
$$
\inf_{\sigma(0)=x,\,\sigma(1)=a}
\int_0^1 \sum_{k\ell} A_{k\ell}\dot\sigma^k(t)\dot\sigma^\ell(t)\,dt
$$
is given by a straight segment, and since this segment is contained inside  $\vp(\mathcal Z_{C''\delta})$, we obtain 
$$
\inf_{\substack{\sigma(0)=x,\, \sigma(1)=a,\\ \sigma\subset \vp(\mathcal Z_{C''\delta})}}
\int_0^1 \sum_{k\ell} A_{k\ell}\dot\sigma^k(t)\dot\sigma^\ell(t)\,dt=\<A(x-a),x-a\>,
$$
which proves 
$$
d\bigl(\vp^{-1}(x),\vp^{-1}(a)\bigr)^2 \leq \bigl(1+\hat C \delta\bigr)\<A(x-a),x-a\>.
$$
The lower bound is proved analogously using that
$$
\sum_{k\ell}g_{k\ell}(z)v^kv^\ell  \geq (1-\hat C\delta) \sum_{k\ell}A_{k\ell}v^kv^\ell \qquad \forall\,z  \in \vp(\mathcal Z_{C''\delta}),\qquad \forall\,v \in \br^d,
$$
concluding the proof.
\end{proof}

Applying now this lemma, we can estimate $V_{N,r}(\mu)$ both from above and below.
Since the argument in both cases is completely analogous, we just prove the upper bound.

Notice that, by the Lipschitz regularity of the metric
and the fact that $\det g_{k\ell}$ is bounded away from zero, we have
$$
\sqrt{\det\,g_{k\ell}(x)} \leq (1+C\delta) \sqrt{\det\,g_{k\ell}(p)}=(1+C\delta) \sqrt{\det\,A}\qquad \forall\,x \in Q.
$$
Combining this estimate with  \eqref{eq:VNcube} and Lemma \ref{lem:metric}, we get
\begin{align*}
V_{N,r}(\mu) &\leq (1+C'\delta) \,\lambda \inf_\alpha \int_Q \min_{a \in \alpha}  \<A(x-a),x-a\>^{r/2} \,\sqrt{\det\,A} \,dx\\
&=(1+C'\delta) \,\lambda \inf_\alpha \int_{A^{1/2}(Q)} \min_{a \in \alpha}  |z-a|^{r} \,dz,
\end{align*}
where $|\cdot|$ denotes the Euclidean norm.

We now apply Theorem \ref{thm: R^d} to the probability measure $\frac{1}{|{A^{1/2}(Q)}|}\one_{A^{1/2}(Q)}\,dz$ to get
\begin{align*}
&\limsup_{N\to \infty}N^{r/d}V_{N,r}(\mu)\\
& \leq (1+C'\delta)\,\lambda\,Q_r\bigl([0,1]^d\bigr)\,\biggl\| \frac{1}{|{A^{1/2}(Q)}|}\one_{A^{1/2}(Q)}\,dz\biggr\|_{L^{d/(d+r)}}\,|{A^{1/2}(Q)}|\\
&=(1+C'\delta)\,\lambda\,Q_r\bigl([0,1]^d\bigr)\,|{A^{1/2}(Q)}|^{(d+r)/d}.
\end{align*}
Observing that
\begin{multline*}
|A^{1/2}(Q)| =\int_Q \sqrt{\det A}\,dx \leq (1+C\delta)\int_Q \sqrt{\det g_{k\ell}(x)}\,dx \\= (1+C\delta)\,\vol(\vp^{-1}(Q)) =(1+C\delta)\, \frac1\lambda,
\end{multline*}
we conclude that
\be
\label{eq:local up}
\limsup_{N\to \infty}N^{r/d}V_{N,r}(\mu) \leq \bigl(1+\bar C \,\delta\bigr)\,Q_r\bigl([0,1]^d\bigr)\,\vol(\vp^{-1}(Q))^{r/d}.
\ee
Arguing similarly for the lower bound, we also have
\be
\label{eq:local down}
\liminf_{N\to \infty}N^{r/d}V_{N,r}(\mu)
\geq (1-\bar C\delta)\,Q_r\bigl([0,1]^d\bigr)\,\vol(\vp^{-1}(Q))^{r/d},
\ee
which concludes the local analysis of the quantization error
for $\mu$ as in \eqref{eq:mu 1}.\\

In the next two sections we will apply these bounds to study $V_{N,r}(\mu)$
for measures of the form  $\mu=\sum_{ij}\alpha_{ij}\mu_{ij}$ where $\alpha_{ij}\neq 0$
for at most finitely many indices, and
 $\mu_{ij}$  has constant density on $\vp_i^{-1}(Q_{i,j})$.

\subsection{Upper bound for $V_{N,r}$}

We consider a compactly supported measure $\mu=\sum_{ij}\alpha_{ij}\mu_{ij}$ where $\alpha_{ij}\neq 0$
for at most finitely many indices, and
 $\mu_{ij}$  is of the form
$\lambda_{ij}\one_{\vp_i^{-1}(Q_{i,j})}d\,\vol$ with 
$$
\vp_i^{-1}(Q_{i,j}) \cap \vp_{i'}^{-1}(Q_{j'}) = \emptyset, \qquad \forall\, i, i',\quad  \forall\,  j \neq j',
$$ 
and $\lambda_{ij}:=\frac{1}{\vol(\vp_i^{-1}(Q_{i,j}))}$ (so that each measure $\mu_{ij}$ has mass $1$).

To estimate $V_{N,r}(\mu)$ we first observe that, for any choice of $N_{ij}$ such that $\sum_{ij}N_{ij}\leq N$ 
the following inequality holds:
$$
V_{N,r}(\mu)\leq 
\sum_{ij}\a_{ij}\,V_{N_{ij},r}(\mu_{ij}).
$$
Indeed, if for any $i,j$ we consider a family of $N_{ij}$ points $\beta_{ij}$
which is optimal for $V_{N_{ij},r}(\mu_{ij})$, the family $\beta:=\cup_{ij}\beta_{ij}$
is an admissible competitor for $V_{N,r}(\mu)$, hence
\begin{align*}
V_{N,r}(\mu)&\leq \int_{\M} \min_{b \in \beta} d(x,b)^r\,d\mu
= \sum_{ij} \a_{ij} \int_{\M}\min_{b \in \beta} d(x,b)^r\,d\mu_{ij}\\
&\leq \sum_{ij} \a_{ij} \int_{\M}\min_{b \in \beta_{ij}} d(x,b)^r\,d\mu_{ij}
=\sum_{ij}\a_{ij}\,V_{N_{ij},r}(\mu_{ij}).
\end{align*}
We want to chose the $N_{ij}$ in an optimal way. As it will be clear from the estimates below, the best choice is to set
\footnote{Notice that, if we were on $\br^d$ and $\vp_i$ are just the identity map,
then the formula for $t_{ij}$ simplifies to 
$$
t_{ij}=\frac{\left(\alpha_{ij}\right)^{d/(d+r)}}
{\sum_{k\ell}\left(\alpha_{k\ell}\right)^{d/(d+r)}},
$$
that is the exact same formula used in \cite[Proof of Theorem 6.2, Step 2]{GL}.
}
$$
t_{ij}:=\frac{\left(\alpha_{ij}\,\vol(\vp_i^{-1}(Q_{i,j}))^{r/d}\right)^{d/(d+r)}}
{\sum_{k\ell}\left(\alpha_{k\ell}\,\vol(\vp_k^{-1}(Q_\ell))^{r/d}\right)^{d/(d+r)}},
$$
and define
$$
N_{ij}:=[t_{ij}N].
$$
Notice that $N_{ij}$ satisfy $\sum_{ij}N_{ij}\leq N$ and
$$
\sum_{ij}\frac{N_{ij}}{N} \to 1 \qquad \text{as $N\to \infty$}.
$$
We observe that each measure $\mu_{ij}$ is a probability measure supported in only one ``cube'' with constant density. Hence we can apply the local quantization error \eqref{eq:local up} to each measure $\mu_{ij}$ to get that
$$
\limsup_{N_{ij} \to \infty}N_{ij}^{r/d}V_{N_{ij},r}(\mu_{ij}) \leq (1+\bar C\delta)\,
Q_r\bigl([0,1]^d\bigr)\,\vol(\vp_i^{-1}(Q_{i,j}))^{r/d}.
$$
Recalling our choice of $N_{ij}$,
\begin{align*}
\limsup_{N\to \infty} N^{r/d}V_{N,r}(\mu) 
&\leq \limsup_{N\to \infty} \sum_{ij}\a_{ij}\,\biggl(\frac{N}{N_{ij}} \biggr)^{r/d}\,N_{ij}^{r/d}V_{N_{ij},r}(\mu_{ij})\\
&\leq (1+\bar C\delta)\,Q_r\bigl([0,1]^d\bigr)\,\sum_{ij}\a_{ij}t_{ij}^{-r/d}\,\vol(\vp_i^{-1}(Q_{i,j}))^{r/d},
\end{align*}
and observing that 
$$
\sum_{ij}\a_{ij}\,t_{ij}^{-r/d}\,\vol(\vp_i^{-1}(Q_{i,j}))^{r/d}=\biggl(\int_\M\hh^{d/(d+r)}\,d\vol\biggr)^{(d+r)/d},
$$
we get
$$
\limsup_{N\to \infty} N^{r/d}V_{N,r}(\mu) \leq 
(1+\bar C\delta)\,Q_r\bigl([0,1]^d\bigr)\,\biggl(\int_\M \hh^{d/(d+r)}\,d\vol\biggr)^{(d+r)/d}.
$$

\subsection{Lower bound for $V_{N,r}$}
We consider again a compactly supported measure $\mu=\sum_{ij}\alpha_{ij}\mu_{ij}$
where $\alpha_{ij}\neq 0$
for at most finitely many indices, and
 $\mu_{ij}$  is of the form
$\lambda_{ij}\one_{\vp_i^{-1}(Q_{i,j})}d\,\vol$
 with 
$$
\vp_i^{-1}(Q_{i,j}) \cap \vp_{i'}^{-1}(Q_{j'}) = \emptyset, \qquad \forall\, i, i',\quad  \forall\,  j \neq j',
$$ 
and $\lambda_{ij}:=\frac{1}{\vol(\vp_i^{-1}(Q_{i,j}))}$ (so that $\int_{\M}\mu_{ij}=1$).
Fix $\e>0$ with $\e \ll \delta$,  and consider the cubes $Q_{j,\e}$ given by 
$$
Q_{j,\e}:=\{y \in Q_{i,j}:\dist(y, \partial Q_{i,j}) >\e\}.
$$
Also, consider a set $\gamma_{ij}$ consisting of $K_{ij}$ points such that 
$$
\min_{a \in \gamma_{ij}} d(x,a)\leq \inf_{z \in \M\setminus \vp_i^{-1}(Q_{i,j})} d(x,z) \qquad \forall\,x \in \vp_i^{-1}(Q_{j,\e})\text { s.t. $\vp_i^{-1}(Q_{i,j})\cap \supp(\mu)\neq\emptyset$.}
$$
Notice that the property of $\mu$ being compactly supported ensures that
$$
K:=\max \biggl\{K_{ij}: \vp_i^{-1}(Q_{i,j})\cap \supp(\mu)\neq\emptyset \biggr\}<\infty
$$
Then, if $\beta$ is a set of $N$ points optimal for $V_{N,r}(\mu)$ and $\beta_{ij}:=\beta \cap 
\vp_i^{-1}(Q_{j})$,
\be
\label{eq:lower}
\begin{split}
V_{N,r}(\mu)&
=  \int_{\M}\min_{b \in \beta} d(x,b)^r\,d\mu\\
&\geq \sum_{ij} \int_{\vp_i^{-1}(Q_{j,\e})}\min_{b \in \beta \cup \gamma_{ij}} d(x,b)^r\,d\mu\\
&= \sum_{ij} \int_{\vp_i^{-1}(Q_{j,\e})}\min_{b \in \beta_{ij} \cup \gamma_{ij}} d(x,b)^r\,d\mu\\
&=\sum_{ij}\a_{ij}^\e \int_{\vp_i^{-1}(Q_{j,\e})}\min_{b \in \beta_{ij} \cup \gamma_{ij}} d(x,b)^r\,d\mu_{ij}^\e\\
&\geq \sum_{ij}\a_{ij}^\e\,V_{N_{ij}+K_{ij},r}(\mu_{ij}^\e),
\end{split}
\ee
where 
$$
\a_{ij}^\e:=\int_{\V_i\cap \vp_i^{-1}(Q_{j,\e}) }d\mu,\qquad 
\mu_{ij}^\e:=\frac{\one_{\V_i\cap \vp_i^{-1}(Q_{j,\e}) }\,d\vol}{\vol\bigl(\vp_i^{-1}(Q_{j,\e}) \bigr)}, \qquad N_{ij}:=\# \beta_{ij}.
$$
We notice that $\a_{ij}^\e\to \a_{ij}$ as $\e \to 0$.

Let $L:=\liminf_{N \to \infty}N^{r/d}V_{N,r}(\mu)$. Notice that $L<\infty$ by the upper bound proved in the previous step.
Choose a subsequence $N(k)$ such that
$$
N(k)^{r/d}V_{N(k),r}(\mu) \to L\qquad \text{as $k\to \infty$}
$$
and, for all $i,j$,
$$
\frac{N_{ij}(k)}{N(k)} \to v_{ij} \in [0,1] \qquad \text{as $k\to \infty$}
$$
Since $\sum_{ij}N_{ij}(k)=N(k)$ we have $\sum_{ij}v_{ij}=1$. 

Moreover $N_{ij}(k) \to \infty$ for every $i,j$.
Indeed, if not,  there would exists $\bar i,\bar j$ such that $N_{\bar i\bar j}(k)+K_{\bar i\bar j}(k)$ would be bounded by a number $M$. 
Hence, since one cannot approximate the absolutely continuous
measure $\mu_{ij}^\e$ only with a finite number $M$ of points, it follows that 
$$
c_0:=V_{M,r}(\mu_{ij}^\e) >0,
$$
that implies in particular 
$$
V_{N_{\bar i\bar j}(k)+K_{\bar i\bar j}(k),r}(\mu_{ij}^\e) \geq c_0>0 \qquad \forall\,k \in \mathbb N
$$
(since $N_{\bar i\bar j}(k)+K_{\bar i\bar j}(k) \leq M$).
This
is impossible as \eqref{eq:lower} would give
$$
L=\lim_{k\to \infty}N(k)^{r/d}V_{N(k),r}(\mu)\geq \lim_{k\to \infty}N(k)^{r/d} \a_{\bar i\bar j}^\e\, c_0=\infty,
$$
which contradicts the finiteness of $L$.

Thanks to this fact, we can now apply the local quantization error \eqref{eq:local down} to deduce that
\begin{multline*}
\liminf_{k \to \infty}N_{ij}(k)^{r/d}V_{N_{ij}(k)+K_{ij}(k),r}(\mu_{ij}^\e)\\
=\liminf_{k \to \infty}\bigl(N_{ij}(k)+K_{ij}(k)\bigr)^{r/d}V_{N_{ij}(k)+K_{ij}(k),r}(\mu_{ij}^\e)\\
\geq (1-\bar C\delta)\,Q_r\bigl([0,1]^d\bigr)\,\vol(\vp_i^{-1}(Q_{j,\e}))^{r/d},
\end{multline*}
which implies that (recalling \eqref{eq:lower})
$$
L \geq (1-\bar C\delta)\,Q_r\bigl([0,1]^d\bigr)\,\sum_{ij}\alpha_{ij}^\e\,v_{ij}^{-r/d}\vol(\vp_i^{-1}(Q_{j,\e}))^{r/d}.
$$
Letting $\e \to 0$ this gives
$$
L \geq (1-\bar C\delta)\,Q_r\bigl([0,1]^d\bigr)\,\sum_{ij}\alpha_{ij}\,v_{ij}^{-r/d}\vol(\vp_i^{-1}(Q_{j}))^{r/d},
$$
and applying \cite[Lemma 6.8]{GL} we finally obtain
\begin{align*}
L &\geq (1-\bar C\delta)\,Q_r\bigl([0,1]^d\bigr)\,\sum_{ij}\Bigl(\alpha_{ij}\vol(\vp_i^{-1}(Q_{j}))^{r/d} \Bigr)\,v_{ij}^{-r/d}\\
&\geq (1-\bar C\delta)\,Q_r\bigl([0,1]^d\bigr)\,\biggl(\sum_{ij}\Bigl(\alpha_{ij}\vol(\vp_i^{-1}(Q_{j}))^{r/d} \Bigr)^{d/(d+r)}\biggr)^{(d+r)/d}\\
&=(1-\bar C\delta)\,Q_r\bigl([0,1]^d\bigr)\,\biggl(\int_{\M}\hh^{d/(d+r)}\,d\vol\biggr)^{(d+r)/d}.
\end{align*}

\subsection{Approximation argument: general compactly supported measures}
In the previous two sections we proved that if $\mu$ is compactly supported and it is of the form
$$
\mu=\sum_{ij}\alpha_{ij}\frac{\one_{\vp_i^{-1}(Q_{i,j})}}{\vol\bigl(\vp_i^{-1}(Q_{i,j})\bigr)}\,d\,\vol
$$
where $Q_{i,j}$ is a family of cubes in $\br^d$  of size at most $\delta$ and $\alpha_{ij}\ne 0$ for finitely many indices, then
\begin{multline}
\label{eq:up low}
(1-\bar C\delta)\,Q_r\bigl([0,1]^d\bigr)\,\biggl(\int_{\M}\hh^{d/(d+r)}\,d\vol\biggr)^{(d+r)/d}
\leq \liminf_{N\to\infty} N^{r/d}V_{N,r}(\mu)\\
\leq \limsup_{N\to\infty} N^{r/d}V_{N,r}(\mu)
\leq (1+\bar C\delta)\,Q_r\bigl([0,1]^d\bigr)\,\biggl(\int_{\M}\hh^{d/(d+r)}\,d\vol\biggr)^{(d+r)/d}.
\end{multline}
To prove the quantization result for general measures with compact support, we need three approximation steps.

First, given a compactly supported measure $\mu=\hh\,d\vol$, we can approximate it with a sequence $\{\mu_k\}_{k \in \mathbb N}$ of measures as above where the size of the cubes $\delta_k \to 0$,
and this allows us to prove that
\be
\label{eq:main}
N^{r/d}V_{N,r}(\mu) \to Q_r\bigl([0,1]^d\bigr)\,\biggl(\int_{\M}\hh^{d/(d+r)}\,d\vol\biggr)^{(d+r)/d}
\ee
for any compactly supported measure of the form $\hh\,d\vol$.
Then, given a singular measure with compact support $\mu=\mu^s,$ we show that 
$$
N^{r/d}V_{N,r}(\mu) \to 0.
$$
Finally, given an arbitrary measure with compact support $\mu=\hh\,d\vol+\mu^s$, we show that \eqref{eq:main} still holds true.

The proofs of these three steps is performed in detail in  \cite[Theorem 6.2, Step 3, Step 4, Step 5]{GL} for the case of $\br^d$.
As it can be easily checked, such a proof applies immediately also in our case, so we will not repeat here
for the sake of conciseness.

This concludes the proof of Theorem \ref{thm:main} when $\mu$ is compactly supported
(in particular, whenever $\M$ is compact).

\section{Proof of Theorem \ref{thm:main}: the non-compact case}
\label{sec:non-compact}

The aim of this section is to study the case of non-compactly supported measures. As we shall see, this situation is very different with respect to the flat case as we need to deal with the growth at infinity of $\mu$.

To state our result, let us recall the notation we already presented in the introduction:
given a point $x_0 \in \M$, we can 
consider polar coordinates $(\r,\vte)$ on $T_{x_0}\M\simeq \br^d$ induced by the constant metric $g_{x_0}$,
where $\vartheta$ denotes a vector on the unit sphere $\S^{d-1}$.
Then we define the quantity $A_{x_0}(\rho)$
as in \eqref{eq:Arho}. 
Our goal is to prove the following result which implies Theorem \ref{thm:main}.
\begin{thm}\label{thm:noncpt}
Let $(\M,g)$ be a complete Riemannian manifold, and 
 let $\mu=\hh \,d\vol+\mu^s$ be a probability measure on $\M$. Then, for any $x_0 \in \M$ and $\delta>0$,
 there exists a constant $C=C(\delta)>0$ such that
\be
\label{eq:bound VNr}
N^r V_{N^d,r}(\mu)\le C \biggl(1+ \int_{\M} d(x,x_0)^{r+\delta}\,d\mu(x)+  \int_{\M} A_{x_0}\bigl({d(x,x_0)}\bigr)^r  \,d\mu(x)\biggr).
\ee
In particular, if there exists a point $x_0 \in \M$ and $\delta>0$ for which the right hand side is finite, we have
\be
\label{eq:quantiz noncpt}
N^{r/d}V_{N,r}(\mu) \to Q_r\bigl([0,1]^d\bigr)\,\biggl(\int_{\M}\hh^{d/(d+r)}\,d\vol\biggr)^{(d+r)/d}.
\ee
\end{thm}

\subsection{Proof of Theorem \ref{thm:noncpt}}
We begin by the proof of \eqref{eq:bound VNr}.
For this we will need the following result, whose proof is contained in \cite[Lemma 6.6]{GL}.
\begin{lem}
\label{lem:1D}
Let $\nu$ be a probability measure on $\br$. Then
\be
\label{eq:finite moment}
N^rV_{N,r}(\nu)\le C \bigg(1+\int_\br |t|^{r+\delta}\,d\nu(t) \bigg).
\ee
\end{lem} 
To simplify the notation, given $v \in T_{x_0}\M$ we use $|v|_{x_0}$ to denote $\sqrt{g_{x_0}(v,v)}$.\\

In order to construct a family of $N^d$ points on $\M$, we argue as follows:
first of all we consider polar coordinates $(\r,\vte)$ on $T_{x_0}\M\simeq \br^d$ induced by the constant metric $g_{x_0}$,
where $\vartheta$ denotes a vector on the unit sphere $\S^{d-1}$,
and then we 
consider a family of ``radii'' $0<\rho_1< \ldots < \rho_N<\infty$ and a set of $N^{d-1}$ points  $\{\vte_1, \ldots, \vte_{N^{d-1}}\} \subset \S^{d-1}$ distributed in a ``uniform'' way on the sphere so that
\be
\label{eq:uniform}
\underset{k}\min\, d_\te(\vte, \vte_k) \le \frac{C}{N} \qquad \forall\, \vte \in \S^{d-1},
\ee
where $d_\te(\vte, \vte_k)$ denotes the distance on the sphere induced by $g_{x_0}$.

We then define the family of points $p_{i,k}$ on the tangent space $T_{x_0}\M$ that, in polar coordinates, are given by $p_{i,k}:=(\rho_i, \vte_k)$, and
we take the family of points on $\M$ given by
$$
x_{i,k}:=\exp_{x_0}(p_{i,k}) \qquad i=1, \ldots, N; \quad k=1, \ldots, N^{d-1}.
$$
We notice the following estimate: 
given a point $x \in \M$, we consider the vector $p=(\rho,\vte) \in T_{x_0}\M$ defined as $p:=\dot{\gamma}(0)$ where $\gamma:[0,1] \ra \M$ is a constant speed minimizing geodesic. By the definition of the exponential map we notice that $x=\exp_{x_0}(p)$
and $\rho=|p|_{x_0}=d(x,x_0)$.
Then, we  can estimate the distance between $x:=\exp_{x_0}(p)$
and $x_{i,k}$ as follows: first 
we consider $\sigma:[0,1]\to\S^{d-1}\subset T_{x_0}\M$ a geodesic (on the unit sphere) connecting  $\vte$ to $\vte_k$
and we define $\eta:=\exp_{x_0}(\rho \,\sigma)$,
and then we connect $\exp_{x_0}\bigl((\r,\vte_k)\bigr)$ to $x_{i,k}$
considering $\gamma|_{[\rho,\rho_i]}$, where $\gamma(s):=\exp_{x_0}\bigl((s,\vte)\bigr)$ is a unit speed geodesic (see Figure \ref{fig2}). 

\begin{figure}[h]
\centerline{\includegraphics[scale=0.5]{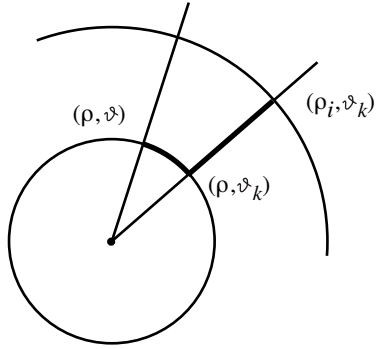}}
\caption{\label{fig2} The bold curve joining $(\r,\vte)$ and $(\r_i,\vte_k)$ provides an upper bound for the distance between the two points.}
\end{figure}
Setting $\eta:=\exp_{x_0}(\rho\,\sigma)$, 
this gives the bound
\begin{align*}
d(x,x_{i,k}) &\leq \int_0^1 \bigl|\dot\eta(t)\bigr|_{\eta(t)}\,dt
+\biggl|\int_{\rho}^{\rho_i}\bigl|\dot\g(s)\bigr|_{\gamma(s)}\,ds\biggr|\\
&= \rho \int_0^1\bigl|d_{\rho \sigma(t)}\exp_{x_0}[\dot\sigma(t)]\bigr|_{\eta(t)}\,dt+|\rho-\rho_i|\\
&\leq A_{x_0}(\rho) \int_0^1\bigl|\dot\sigma(t)\bigr|_{x_0}\,dt+|d(x,x_0)-\rho_i|\\
&=A_{x_0}\bigl({d(x,x_0)}\bigr) \,d_\theta(\vte_k,\vte)+|d(x,x_0)-\rho_i|,
\end{align*}
where $A_{x_0}(\rho)$ is defined in \eqref{eq:Arho},
and we used that $\sigma(t)$ is a geodesic (on the sphere) from $\vte_k$ to $\vte$ and that $\rho=d(x,x_0)$.

Notice that, thanks to the estimate above and by \eqref{eq:uniform}, 
\begin{align*}
 \underset{i,k}\min\, d(x, x_{i,k})^r &\leq 
 \underset{i,k}\min \bigg[ A_{x_0}\bigl({d(x,x_0)}\bigr)\, d_\theta(\vte,\vte_k)+|d(x,x_0)-\rho_i| \bigg]^r\\
& \leq \underset{i}\min \bigg[ A_{x_0}\bigl({d(x,x_0)}\bigr)\, \frac{C}{N}+|d(x,x_0)-\rho_i| \bigg]^r.
\end{align*}
We can now estimate the quantization error:
\begin{align*}
N^r V_{N^{d},r}(\mu)& \le N^r \int_{\M} \underset{i,k}\min\, d(x, x_{i,k})^r\,d\mu(x)\\
&\le N^r \int_{\M}  \underset{i,k}\min \bigg[ A_{x_0}\bigl({d(x,x_0)}\bigr)\, \frac{C}{N}+|d(x,x_0)-\rho_i| \bigg]^r\,d\mu(x).
\end{align*}
Using that $(a+b)^r\le 2^{r-1}(a^r+b^r)$ for $a,b>0$ we get
\begin{align*}
N^r V_{N^{d},r}(\mu)& \le N^r 2^{r-1} \Bigg[\int_{\M}  \underset{i}\min\,|d(x,x_0)-\rho_i|^r\,d\mu(x)+ \int_{\M}A_{x_0}\bigl({d(x,x_0)}\bigr)^r  \biggl(\frac{C}{N}\biggr)^r\,d\mu(x)\Bigg]\\
&=N^r 2^{r-1} \int_\M   \underset{i}\min\,|d(x,x_0)-\rho_i|^r\,d\mu(x)
+ C^r2^{r-1} \int_{\M} A_{x_0}\bigl({d(x,x_0)}\bigr)^r  \,d\mu(x).
\end{align*}
Let us now consider the map $d_{x_0}:\M \to \br$ defined as $d_{x_0}(x):=d(x,x_0)$, and define the probability measure on $\br$
given by $\mu_1:=(d_{x_0})_\#\mu.$ In this way
$$
\int_\M   \underset{i}\min\,|d(x,x_0)-\rho_i|^r\,d\mu(x)
=\int_{\br} \underset{i}\min\,|s-\rho_i|^r\,d\mu_1(s).
$$
We now choose the radii $\rho_i$ to be optimal for the quantization problem in one dimension for $\mu_1$.
Then the above estimate and Lemma \ref{lem:1D} yield
\begin{align*}
N^r V_{N^{d},r}(\mu)& \le N^r 2^{r-1} V_{N,r}(\mu_1)+ C^r2^{r-1} \int_{\M} A_{x_0}\bigl({d(x,x_0)}\bigr)^r  \,d\mu(x),\\
& \le C'\biggl(1+ \int_0^\infty s^{r+\delta}\,d\mu_1(s)+ \int_{\M} A_{x_0}\bigl({d(x,x_0)}\bigr)^r  \,d\mu(x)\biggr)\\
& = C'\biggl(1+ \int_{\M} d(x,x_0)^{r+\delta}\,d\mu(x)+  \int_{\M} A_{x_0}\bigl({d(x,x_0)}\bigr)^r  \,d\mu(x)\biggr),
\end{align*}
that concludes the proof of \eqref{eq:bound VNr}.\\

To show why this bound implies \eqref{eq:quantiz noncpt} (and hence Theorem \ref{thm:main}
in the general non-compact case),
we first notice that by \eqref{eq:bound VNr} it follows that, for any $M \geq 1$,
\begin{multline}
\label{eq:bound VNr 2} 
M^{r/d} V_{M,r}(\mu)\\
 \le
C\biggl(1+ \int_{\M} d(x,x_0)^{r+\delta}\,d\mu(x)+  \int_{\M} A_{x_0}\bigl({d(x,x_0)}\bigr)^r  \,d\mu(x)\biggr).
\end{multline}
Indeed, for any $M \geq 1$ there exists $N \geq 1$ such that $N^d \leq M <(N+1)^d$, hence
(since $V_{M,r}$ is decreasing in $M$)
\begin{multline*}
M^{r/d} V_{M,r}(\mu) \leq (N+1)^{r} V_{N^d,r}(\mu)
=\biggl(1+\frac{1}{N}\biggr)^{r} N^{r}V_{N^d,r}(\mu)\\
\le C\biggl(1+ \int_{\M} d(x,x_0)^{r+\delta}\,d\mu(x)+  \int_{\M} A_{x_0}\bigl({d(x,x_0)}\bigr)^r  \,d\mu(x)\biggr),
\end{multline*}
which proves \eqref{eq:bound VNr 2}.

We now prove \eqref{eq:quantiz noncpt}. Observe that, as shown in  \cite[Proof of Theorem 6.2, Step 5]{GL},
once the asymptotic quantization is proved for compactly supported probability measures,
by the monotone convergence theorem one always has
$$
\underset{N\to\infty}{\liminf}\,N^{r/d} V_{N,r}(\mu)\geq Q_r\bigl([0,1]^d\bigr)\,\biggl(\int_{\M}\hh^{d/(d+r)}\,d\vol\biggr)^{(d+r)/d},
$$
hence one only have to prove the limsup inequality.

For that, one splits the measure $\mu$ as the sum of $\mu_R^1:=\chi_{B_R(x_0)}\mu$
and $\mu_R^2:=\chi_{\M\setminus B_R(x_0)}\mu$, where $R \gg 1$.
Then one applies \cite[Lemma 6.5(a)]{GL} to bound from above $N^{r/d} V_{N,r}(\mu)$
in terms of $N^{r/d} V_{N,r}(\mu_R^1)$ and $N^{r/d} V_{N,r}(\mu_R^2)$, and uses 
the result in the compact case for $N^{r/d} V_{N,r}(\mu_R^1)$,
to obtain that, for any $\e \in (0,1)$
\begin{multline*}
\underset{N\to\infty}{\limsup}\,N^{r/d} V_{N,r}(\mu)\leq (1-\e)^{-r/d}\,Q_r\bigl([0,1]^d\bigr)\,\biggl(\int_{B_{R}(x_0)}\hh^{d/(d+r)}\,d\vol\biggr)^{(d+r)/d}\\
+\mu(\M\setminus B_R(x_0))\,\e^{-r/d}\,\underset{N\to\infty}{\limsup}\,N^{r/d} V_{N,r}\biggl(\frac{1}{\mu(\M\setminus B_R(x_0))}\mu_R^2\biggr).
\end{multline*}
Thanks to \eqref{eq:bound VNr 2}, we can bound the limsup in the right hand side by
$$
\e^{-r/d}\biggl(\mu(\M\setminus B_R(x_0))+ \int_{\M} d(x,x_0)^{r+\delta}\,d\mu_R^2(x)+  \int_{\M} A_{x_0}\bigl({d(x,x_0)}\bigr)^r  \,d\mu_R^2(x)\biggr),
$$
that tends to $0$ as $R\to \infty$ by dominated convergence.
Hence, letting $R\to \infty$ we deduce that 
\begin{align*}
\underset{N\to\infty}{\limsup}\,N^{r/d} V_{N,r}(\mu)&\leq (1-\e)^{-r/d}\,Q_r\bigl([0,1]^d\bigr)\,\lim_{R\to \infty}\biggl(\int_{B_{R}(x_0)}\hh^{d/(d+r)}\,d\vol\biggr)^{(d+r)/d}\\
&=(1-\e)^{-r/d}\,Q_r\bigl([0,1]^d\bigr)\,\biggl(\int_{\M}\hh^{d/(d+r)}\,d\vol\biggr)^{(d+r)/d},
\end{align*}
and the result follows letting $\e \to 0$.


\section{Proof of Theorem \ref{thm:counter}}
\label{sec:counter}
We begin by noticing that if
$$
\int_{\H^2} d(x,x_0)^p\,d\mu<\infty
$$
for some $x_0 \in \H^2$, then this holds for any other point: indeed, given $x_1\in\H^2$,
$$
\int_{\H^2} d(x,x_1)^p\,d\mu\leq 2^{p-1} 
\int_{\H^2} \bigl[d(x,x_0)^p+d(x_0,x_1)^p]\,d\mu<\infty.
$$
In particular, it suffices to check the moment condition at only one point.

We fix a point $x_0 \in \H^2$ and we use the exponential map at $x_0$ to identify $\H^2$
with $(\br^2,d^2\rho+\sinh \rho\,d^2\vte)$.
Then, we define the measure 
$$
\mu:=\sum_{k \in \N} e^{-(1+\e)k}\HH^1{\llcorner{\S^1_k}},
$$
where $\HH^1\llcorner{\S^1_R}$ denotes the $1$-dimensional Haudorff measure restricted to the circle around the origin of radius $R$, and $\e>0$ is a constant to be fixed.

We begin by noticing that
\begin{multline*}
\int_{\H^2} d(x,x_0)^p\,d\mu=\sum_{k \in \N}e^{-(1+\e)k}\int_{\S^1_k} \rho^p\,d\HH^1\\
=\sum_{k \in \N}e^{-(1+\e)k}k^p\,2\pi\sinh(k) \approx \sum_{k \in \N}e^{-\e k}k^p <\infty
\end{multline*}
for all $p>0$.

An important ingredient of the proof will be the following estimate on the quantization 
error for the uniform measure on a circle around the origin.

\begin{lem}
\label{lem:circle}
For any $R\geq 1$ and $M \in \mathbb N$ we have
$$
V_{M,r}\bigl(\HH^1\llcorner\S^1_R\bigr) \gtrsim \biggl(\frac{e^R}{2R}-M \biggr)_+R.
$$
\end{lem}
\begin{proof}
To prove the above estimate, we built a good competitor for the minimization problem. Let us denote with $[\cdot]$ the integer part, and define
$$
L:=\[\frac{e^R}{2R} \].
$$ 
We split $\S^1_R$ in $2L$ arcs  $\Sigma_{i,R}$ of equal length.
Notice that the following estimate holds: there exists a positive constant $c$, independent of $R,$ such that
\be
\label{eq:arcs}
d(\Sigma_{2j,R},\Sigma_{2j',R})> c \qquad \forall\ j\neq j' \in \{1, \ldots, L \} .
\ee
To show this fact, one argues as follows: consider a geodesic connecting a point $x_1 \in \Sigma_{2j,R}$  to $x_2 \in \Sigma_{2j',R}$.
Because $j \neq j'$ any curve connecting them has to rotate by an angle of order at least $R/e^R$.
Now, two cases arise: either the geodesic $\gamma:[0,1]\to \H^2$ is always contained inside $\br^d \setminus B_{R-1}(0)$, or not.
In the first case we exploit that the metric is always larger than 
$\sinh^2(R-1)d^2\vte$. More precisely, if we denote by $(e_\rho,e_\theta)$ a basis of tangent vectors in polar coordinates
\begin{align*}
d(x,y)&=\int_0^1 \sqrt{g_{\gamma(t)}\bigl(\dot\g(t),\dot\g(t)\bigr)}\,dt \\
&=\int_0^1 \sqrt{\bigl(\dot\g(t)\cdot e_\rho\bigr)^2+\sinh^2(\rho) {\bigl(\dot \gamma(t)\cdot e_\theta\bigr)}}\,dt\\
&\geq \sinh(R-1)\int_0^1|\dot \gamma(t)\cdot e_\theta|\,dt \gtrsim e^{R-1} \frac{R}{e^R} \approx R \geq 1,
\end{align*}
where for the last inequality we used that $\gamma$ has to rotate by an angle of order at least $R/e^R$.
In the second case, to enter inside the ball $B_{R-1}(0)$ the geodesic has to travel a distance at least $1$, so its length is greater that $1$.
This proves the validity of \eqref{eq:arcs}.

We pick now a family of $M$ points $\{ x_\ell\}_{\ell=1}^M.$ Then, by \eqref{eq:arcs} and triangle inequality, we have that for every index $\ell$ there exists at most one index $j(\ell)$ such that
$$
d(x_\ell, \Sigma_{2j,R})> \frac{c}{2} \qquad \forall\ j\neq j(\ell).
$$
Therefore there exists a family of indices $J \in \{1, \ldots, L \}$ of cardinality at least $(L-M)_+$  such that
$$
d(x_\ell, \Sigma_{2j,R})> \frac{c}{2} \qquad \forall j\in J, \quad \forall \ell=1, \ldots, M.
$$
We can now estimate the quantization error:
\begin{align*}
V_{M,r}\bigl(\HH^1\llcorner\S^1_R\bigr)&= \underset{\alpha \subset \H^2 : |\alpha|=M}\min \int_{\S^1_R}\underset{x_\ell \in \alpha}\min\,d(x,x_\ell)^r\,d\HH^1\\
&\ge \underset{\alpha \subset \H^2 : |\alpha|=M}\min \sum_{j=1}^L\int_{\Sigma_{2j, R}}\underset{x_\ell \in \alpha}\min\,d(x,x_\ell)^r\,d\HH^1\\
&\ge\sum_{j\in J}\int_{\Sigma_{2j, R}} \(\frac{c}{2}\)^r\,d\HH^1\gtrsim (L-M)_+R,
\end{align*}
where at the last step we used that $\HH^1(\Sigma_{2j, R})\approx R.$

\end{proof}

We can now conclude the proof.
Indeed, given a set of points $\{x_\ell\}_{1\leq \ell \leq N^2}$ optimal for $\mu$, these points are admissible for the quantization problem of each measure $\HH^1{\llcorner{\S^1_k}}$, therefore
\begin{align*}
V_{N^2,r}(\mu)&=\sum_{k \in \N} e^{-(1+\e)k}\int_{\S^1_k}\min_\ell d(x,x_\ell)^r\,d\HH^1(x)\\
&\geq \sum_{k \in \N} e^{-(1+\e)k}V_{N^2,r}\bigl(\HH^1\llcorner\S^1_k\bigr)\\
&\gtrsim \sum_{k \in \N} e^{-(1+\e)k}\biggl(\frac{e^k}{2k}-N^2 \biggr)_+k,
\end{align*}
where at the last step we used Lemma \ref{lem:circle}.
Noticing that, for $N$ large,
$$
\frac{e^k}{2k}-N^2 \geq \frac14 \frac{e^k}{k} \qquad \text{for $k \geq \log(N^4)$},
$$
we conclude that
\begin{align*}
N^rV_{N^2,r}(\mu)&\gtrsim \frac{N^r}{4} \sum_{k \geq \log(N^4)} e^{-(1+\e)k} \frac{e^k}{k}\,k\\
&=\frac{N^r}{4} \sum_{k \geq \log(N^4)} e^{- \e k}\\
&\gtrsim N^r \int_{\log(N^4)}^\infty e^{-\e t} \,dt \approx \frac{N^r N^{-4 \e}}{\e} \to \infty
\end{align*}
as $N\to \infty$ provided we choose $\e<r/4$.

\bigskip

{\it Acknowledgments:} The author is grateful to Beno\^{i}t Kloeckner for useful comments on this paper.


\begin{thebibliography}{999999}

%
%
%
%
 
\bibitem{BJR} G. Bouchitt\'e, C. Jimenez, R. Mahadevan, {\em Asymptotic analysis of a class of optimal location problems}, J. Math. Pures Appl. (9) 95 (2011), no. 4, 382-419.
\bibitem{BBSS}A. Brancolini, G. Buttazzo, F. Santambrogio, E. Stepanov, {\em Long-term planning versus short-term planning in the asymptotical location problem}, ESAIM Control Optim. Calc. Var. 15 (2009), no. 3, 509-524.
\bibitem{BW} J. Bucklew and G. Wise, {\em Multidimensional Asymptotic Quantization Theory with rth Power Distortion Measures}, IEEE Inform. Theory 28 (2), 239-247, 1982. 
\bibitem{C}  D. L. Cohn, {\em Measure theory}, Birkh\"auser, Boston, Mass., 1980.
%
%
%

\bibitem{GL} S. Graf, H. Luschgy, {\em Foundations of Quantization for Probability Distributions}, Lecture Notes in Math. 1730, Springer-Verlag, Berlin Heidelberg, 2000.

\bibitem{B} B. Kloeckner, {\em Approximation by finitely supported measures}, ESAIM Control Optim. Calc. Var. 18 (2012), no. 2, 343-359.

\bibitem{L} J. M. Lee, {\em Riemannian manifolds. An introduction to curvature}, Graduate Texts in Mathematics, 176. Springer-Verlag, New York, 1997.

\bibitem{R} J. G. Ratcliffe, {\em Foundations of hyperbolic manifolds}, Second edition. Graduate Texts in Mathematics, 149. Springer, New York, 2006.

%
%
%
%
%

\end{thebibliography}
\end{document}